\documentclass[a4paper, 11pt]{article}
\usepackage[utf8]{inputenc}
\usepackage[margin=2 cm]{geometry}

\usepackage{amssymb, amstext, amscd, amsmath, color}

\usepackage{url}
\usepackage{tikz}
\usepackage{epstopdf}
\usepackage{amsfonts}
\usepackage{amsthm}
\usepackage{amsfonts}
\usepackage{array}
\usepackage{epsfig}
\usepackage{eucal}
\usepackage{latexsym}
\usepackage{mathrsfs}
\usepackage{textcomp}
\usepackage{verbatim}
\usepackage{setspace}
\usepackage{enumerate}
\usepackage[colorlinks=true,linkcolor=blue,urlcolor=blue,citecolor=red]{hyperref}

\newtheorem{thm}{Theorem}[section]
\newtheorem{lemma}[thm]{Lemma}
\newtheorem{prop}[thm]{Proposition}
\newtheorem{cor}[thm]{Corollary}

\newtheorem*{claim}{Claim}
\newtheorem*{thm*}{Theorem}
\newtheorem*{lemma*}{Lemma}
\newtheorem*{prop*}{Proposition}
\newtheorem*{cor*}{Corollary}
\newtheorem*{conj*}{Conjecture}
\theoremstyle{remark}
\newtheorem{defn}[thm]{Definition}
\newtheorem*{defn*}{Definition}

\newtheorem{rmk}[thm]{Remark}

\DeclareMathOperator{\mlt}{mlt}

\DeclareMathOperator{\gcr}{gcr}
\DeclareMathOperator{\grn}{grn}

\DeclareMathOperator{\rank}{rank}

\newcommand*{\NN}{\mathbb{N}}
\newcommand*{\RR}{\mathbb{R}}

\usepackage{optidef}
\usepackage[shortlabels]{enumitem}

\usepackage{tikz}
\tikzstyle{vertex}=[circle, draw, inner sep=0pt,minimum size=6pt, fill=black]
\newcommand{\vertex}{\node[vertex]}
\usetikzlibrary{calc}
\tikzstyle{vertex}=[circle, draw, fill=black, inner sep=0pt, minimum size=4pt]
\tikzstyle{redvertex}=[circle, draw, red, fill=red, inner sep=0pt, minimum size=4pt]
\tikzstyle{bluevertex}=[circle, draw, cyan, fill=cyan, inner sep=0pt, minimum size=4pt]
\tikzstyle{lnode}=[circle,white,draw, inner sep=1pt, font=\scriptsize]
\tikzstyle{edge}=[line width=1pt]

\begin{document}

\title{Computing maximum likelihood thresholds using graph rigidity}

\author{Daniel Irving Bernstein\thanks{Department of Mathematics,
Tulane University, USA, dbernstein1@tulane.edu}
\and
Sean Dewar\thanks{School of Mathematics, University of Bristol, UK, sean.dewar@bristol.ac.uk}
\and Steven J. Gortler\thanks{School of Engineering 
and Applied Sciences,
Harvard University, USA, sjg@cs.harvard.edu}
\and Anthony Nixon\thanks{Department of Mathematics and Statistics,
Lancaster University, UK, a.nixon@lancaster.ac.uk}
\and Meera Sitharam\thanks{Department of Computer Science,
University of Florida, USA, sitharam@cise.ufl.edu}
\and Louis Theran\thanks{School of Mathematics and Statistics,
University of St Andrews, UK, louis.theran@st-andrews.ac.uk }
}


\date{}
\maketitle

\begin{abstract}
The maximum likelihood threshold (MLT) of a graph $G$ is the minimum number of samples to almost 
surely guarantee existence of the maximum likelihood estimate in the corresponding Gaussian graphical 
model. Recently a new characterization of the MLT in terms of rigidity-theoretic properties of 
$G$ was proved \cite{Betal}. This characterization was then used to give new combinatorial lower bounds on the MLT of any graph.  
We continue this line of research by exploiting combinatorial rigidity results to compute the MLT precisely for several families of graphs. These include graphs with at most $9$ vertices, graphs with at most 24 edges, every graph sufficiently close to a complete graph and graphs with bounded degrees.
\end{abstract}

\medskip


\section{Introduction}

Let $G$ be a graph with $n$ vertices.  The Gaussian graphical model
associated with $G$ is the set of $n$-variate normal distributions 
$\mathcal{N}(0,\Sigma)$ so that if $ij$ is \emph{not} an edge of $G$,
then $(\Sigma^{-1})_{ij} = 0$, i.e.~the corresponding 
random variables are conditionally independent
given all of the other random variables.

A question, originally posed by Dempster \cite{dempster1972covariance},
that has gotten a lot of attention after Uhler's foundational 
work \cite{uhler2012geometry} on the topic is: 
\emph{for a fixed graph $G$, how many datapoints%
\footnote{We assume that the samples are i.i.d.~from a distribution whose probability measure is mutually absolutely 
 continuous with respect to Lebesgue measure.} %
are needed for the maximum likelihood estimator of the associated Gaussian graphical model to exist almost surely?}.  
This minimum number of 
datapoints is called the \emph{maximum likelihood threshold (MLT) of $G$}, 
which we denote $\mlt(G)$.

Efficiently computing the maximum likelihood threshold of an arbitrary 
graph seems out of reach to current techniques.  Instead most of the 
literature focuses on providing combinatorial bounds on $\mlt(G)$.  The 
basic tool is an algebraic graph parameter called the 
\emph{generic completion rank (GCR)} of $G$, denoted $\gcr(G)$.  There 
are various equivalent definitions of $\gcr(G)$.  For our 
purposes, the right one is due to Gross and Sullivant 
\cite{gross2018maximum}, who showed that $\gcr(G) = d+1$, 
where $d$ is the smallest dimension such that a generic 
geometric structure made of rigid bars connected at freely 
rotating joints does not support an equilibrium stress.  (Full 
definitions are given below.)

While $\mlt(K_n) = n$, much lower numbers can be achieved. The following basic lemma is a starting point for an arbitrary graph.

\begin{lemma}\label{lem: mlt monotone}
Let $G$ be a graph and $H$ a subgraph of $G$.  Then $\mlt(H)\le \mlt(G)$.
\end{lemma}

Uhler proved the following upper bound on the MLT.
\begin{thm}[\cite{uhler2012geometry}]\label{thm: uhler gcr mlt}
Let $G$ be a graph.  Then $\mlt(G)\le \gcr(G)$.
\end{thm}

To get lower bounds, the authors, in \cite{Betal}, introduced a new graph parameter
called the \emph{globally rigid subgraph number} of $G$, denoted 
$\grn^*(G)$.  This 
is the largest dimension $d$ so that $G$ contains a globally 
$d$-rigid subgraph on at least $d+2$ vertices. They showed that $\mlt(G)\ge \grn^*(G) + 2$. 

Rigorous experiments reported in \cite{Betal} indicate that, 
for sparse Erd\"os-Rényi random graphs, in fact, with 
high probability
\(
    \grn^*(G) + 2 = \gcr(G)
\)
which implies that, whp, $\gcr(G) = \mlt(G)$.
For sufficiently small MLT and GCR, 
this relationship is deterministic.
\begin{thm}[\cite{Betal}]\label{thm: equality of gcr and mlt}
If $G$ is a graph and $\mlt(G) \le 3$ or $\gcr(G) \le 4$, then $\mlt(G) = \gcr(G)$.
\end{thm}
For the general case, however, Blekherman and Sinn \cite{blekherman2019maximum} showed that a ``whp'' 
equality is the best one can hope for.
\begin{thm}[\cite{blekherman2019maximum}]
For complete bipartitute graphs $K_{m,m}$,
\[
    \mlt(K_{m,m}) = o(\gcr(K_{m,m}))\qquad \text{(as $m\to \infty$)}
\]
and, in particular
\[
    4 = \mlt(K_{5,5}) < \gcr(K_{5,5}) = 5.
\]
\end{thm}

\subsection{Contributions}
The conjectural picture is that, for almost all $G$, we have 
$\mlt(G) = \gcr(G)$.  On the other hand, the Blekhmermann--Sinn 
examples show that
the two quantities can be 
arbitrarily far apart.  Against this, 
our Theorem \ref{thm: equality of gcr and mlt}
says that for very small GCR and MLT, the two quantities 
\emph{always coincide}.

It is then natural to ask how large or how sparse a graph $G$ needs
to be in order to have $\mlt(G) < \gcr(G)$.  That is the 
topic we take up here.  

Our first results imply that the Blekherman--Sinn example of $K_{5,5}$
is the smallest possible graph where the MLT and GCR do not coincide.
No graph with fewer vertices can have this property.
\begin{thm}\label{THM:SMALL}
Let $G$ a graph on at most $9$ vertices. Then $\mlt(G) = \gcr(G)$.
\end{thm}
Neither can a graph with fewer edges.
\begin{thm}\label{thm:atmost24}
    Let $G$ be a graph with at most 24 edges.
    Then $\mlt(G) = \gcr (G)$.
\end{thm}

Turning to the general case, we prove that any graph 
which is nearly complete also must have equal MLT and GCR.
\begin{thm}\label{t:manyedges}
    Let $G$ be a graph whose complement has at most 5 edges.
    Then $\mlt(G)=\gcr(G) \geq n-2$.
\end{thm}
So must a graph with sufficiently small minimum and maximum degrees.
\begin{thm}\label{thm:degbounded}
    Let $G$ be a connected graph with minimum degree at most 4 and maximum degree at most 5.
    Then $\mlt(G) = \gcr(G) \leq 5$.
\end{thm}

We believe that Theorem \ref{t:manyedges} might not be 
best possible, in the sense that the same statement may hold with a 
number larger than $5$.  At the other extreme, $K_{5,5}$ shows 
that removing $20$ edges from $K_{10}$ does give a graph with 
different MLT and GCR.

Theorem \ref{THM:SMALL} is proved in Section~\ref{sec: many / few}, with the more technical parts of the proof deferred until Section \ref{app:b}, and Theorem \ref{thm:atmost24} is proved in Section \ref{sec:few edges}. Theorems \ref{t:manyedges} and \ref{thm:degbounded} are proved in Section \ref{sec:manybound}. First we give some preliminary results from graph rigidity and recap some necessary theory from \cite{Betal}.

\section{MLT and rigidity theory}

In this section we introduce the necessary background from rigidity theory, including results on equilibrium stresses and MLT.
Figure~\ref{fig:rigidity} illustrates the following definitions for $d=2$.

\subsection{Rigid graphs}

\begin{defn}\label{def: framework}
Let $d\in \NN$ be a dimension.  A \emph{framework in $\mathbb{R}^d$} is 
an ordered pair $(G,p)$ where $G$ is a graph with $n$ vertices $\{1, \ldots, n\}$ and 
$p = (p(1), \ldots, p(n))$ is a configuration of $n$ points in $\RR^d$.
Two frameworks $(G,p)$ and $(G,q)$ are \emph{equivalent} if 
\[
    \|p(j) - p(i)\| = \|q(j) - q(i)\|\qquad 
    \text{ for all edges $ij$ of $G$}
\]
and \emph{congruent} if $p$ and $q$ are related by a Euclidean isometry,
i.e.~if there exists a Euclidean isometry $T:\mathbb{R}^d\rightarrow\mathbb{R}^d$
such that $q(i)=T(p(i))$ for $i = 1,\dots,n$.
$(G,p)$ is \emph{globally rigid} in dimension $d$ if all equivalent $d$-dimensional 
frameworks are congruent.
If this happens only for some neighborhood $U$ around $p$,
i.e.~if $(G,p)$ and $(G,q)$ are congruent whenever $q \in U$ and $(G,q)$ and $(G,p)$ are equivalent,
then $(G,p)$ is said to be \emph{rigid} in dimension $d$.
\end{defn}

\begin{figure}[h]
    \centering
        \begin{tikzpicture}
        \vertex (1) at (0,0){};
        \vertex (2) at (1,0){};
        \vertex (3) at (1,1){};
        \vertex (4) at (0,1){};
        \path
            (1) edge (2) edge (4)
            (2) edge (3)
            (3) edge (4)
        ;
        \node at (1.7,1/2){$\leftrightarrow$};
        \vertex (1a) at (0+2.3,0){};
        \vertex (2a) at (1+2.3,0){};
        \vertex (3a) at ($ (3/2+2.3,{sqrt(3)/2}) $){};
        \vertex (4a) at ($ (1/2+2.3,{sqrt(3)/2}) $){};
        \path
            (1a) edge (2a) edge (4a)
            (2a) edge (3a)
            (3a) edge (4a)
        ;      
    \end{tikzpicture}
    \qquad\quad
    \begin{tikzpicture}
        \vertex (1c) at (0,0){};
        \vertex (2c) at (1,0){};
        \vertex (3c) at (.3,.9){};
        \vertex (4c) at (.9,.6){};
        \path
            (1c) edge (2c) edge (3c)
            (2c) edge (3c) edge (4c)
            (3c) edge (4c)
        ;
        \node at (1.7,1/2){$\leftrightarrow$};
        \vertex (1c) at (0+2.3,0){};
        \vertex (2c) at (1+2.3,0){};
        \vertex (3c) at (.3+2.3,.9){};
        \vertex (4c) at (.46+2.3,.3){};
        \path
            (1c) edge (2c) edge (3c)
            (2c) edge (3c)
            (4c) edge (3c) edge (2c)
        ;
    \end{tikzpicture}
    \qquad\quad
    \begin{tikzpicture}
        \vertex (a) at (0,0){};
        \vertex (b) at (1,0){};
        \vertex (c) at (0.5,1){};
        \vertex (d) at (0.5,0.5){};
        \path
            (a) edge (b) edge (c) edge (d)
            (b) edge (c) edge (d)
            (c) edge (d)
        ;
    \end{tikzpicture}
    \caption{The framework on the left fails to be rigid because there exist arbitrarily close frameworks that are equivalent but not congruent. The frameworks in the middle are rigid but fail to be globally rigid since they are equivalent but not congruent. Finally, the framework on the right is globally rigid and therefore also rigid.}
    \label{fig:rigidity}
\end{figure}
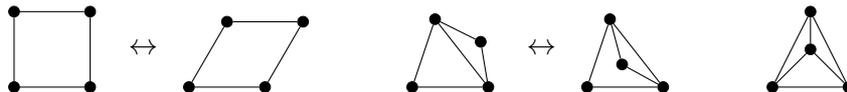

(Global) rigidity of a specific framework is difficult to check \cite{abbot,saxe1980embeddability}, but 
for each dimension $d$, every graph has a generic behavior.
Following \cite{uhler2012geometry,gross2018maximum}, 
we use the following notion of 
generic.
Let $p$ be a configuration of $n$ points in $\RR^d$.  We 
say that $p$ is \emph{generic} if the coordinates of $p$ do not 
satisfy any polynomial with rational coefficients.
Fundamentally, in the generic case, we can treat rigidity and global rigidity
as properties of a graph rather than as properties of a framework. Indeed by \cite{asimow1978rigidity,gortler2010characterizing}, if we let $d$ be a fixed dimension and $G$ a graph, then either 
every generic $d$-dimensional 
framework $(G,p)$ is (globally) rigid or every generic $d$-dimensional framework
$(G,p)$ is not (globally) rigid. Hence we call $G$ \emph{(globally) $d$-rigid} if its generic $d$-dimensional
frameworks are (globally) rigid. 

    Let $G$ be a graph with $n$ vertices and $m$ edges.
    The \emph{rigidity matrix} $R(G,p)$ of a $d$-dimensional framework $(G,p)$ is the $m\times dn$ matrix
    whose rows are indexed by the edges of $G$, columns indexed by the coordinates of $p(1),\dots,p(n)$,
    where the entry corresponding to edge $e$ and $p(v)_i$
    is $p(v)_i-p(u)_i$ if $e = vu$, and $0$ if $v$ is not incident to $e$.
    We call $G$ \emph{$d$-independent}
if the rows of $R(G,p)$ are linearly independent for a (or any) generic framework $(G,p)$.
In Figure~\ref{fig:rigidity}, the graphs underlying the frameworks in the middle and on the left are $2$-independent, whereas the graph of the framework on the right is not. 

Given vertices $i$ and $j$ of a graph $G$, we write $i\sim j$ to indicate that $G$ has an edge between $i$ and $j$.
Let $G$ be a graph with $n$ vertices.  Let $(G,p)$ be a framework.  An \emph{equilibrium stress}
$\omega$ of $(G,p)$ is 
an assignment of  weights $\omega_{ij}$ to the edges of $G$ so that,
for all vertices $i$
\[
    \sum_{j\sim i} \omega_{ij}(p(j) - p(i)) = 0 \qquad \text{(sum over neighbors of $i$).}
\]
The \emph{equilibrium stress matrix} associated to an equilibrium stress $\omega$ is the 
matrix $\Omega$ obtained by setting $\Omega_{ji} = \Omega_{ij} = -\omega_{ij}$ for all edges $ij$ of $G$,
$\Omega_{ii} = \sum_{j} \omega_{ij}$ and all other entries zero.
The \emph{rank} and \emph{signature} of $\omega$ are defined to be the rank 
and signature of $\Omega$, and $\omega$ is said to be \emph{PSD} if $\Omega$ is positive semi-definite.

Building on \cite{uhler2012geometry,gross2018maximum}, the following key result was proved in \cite{Betal} which reformulated the MLT of a graph in terms of \emph{equilibrium 
stresses}.

\begin{thm}\label{thm: main mlt stress}
Let $G$ be a graph with $n$ vertices. Then the MLT of $G$ is 
$d+1$ if and only if $d$ is the smallest dimension in which 
no generic $d$-dimensional framework supports non-zero 
PSD equilibrium stress.
\end{thm}

We will make repeated use of the following theorem which will allow us to determine the MLT when $G$ is globally $d$-rigid.

\begin{thm}[\cite{cgt1}]\label{thm: gen ur}
Let $G$ be a graph with $n\ge d + 2$ vertices and $d$ a dimension.  If $G$ globally $d$-rigid, then there is a generic framework $(G,p)$ with a 
PSD equilibrium stress of rank $n - d - 1$.
\end{thm}

The theorem will often be applied through the following corollary. $G$ is a \emph{$d$-circuit} if it is not $d$-independent, but every proper subgraph is.

\begin{cor}[\cite{Betal}]
\label{cor:gr+cir}
    Let $G$ be a $d$-circuit.  Then $\gcr(G)=d+2$.
    Further, if $G$ is globally $d$-rigid
    then $\mlt(G) = d+2$.
\end{cor}

\subsection{Graph operations}

    A \emph{($d$-dimensional) $0$-extension} of a graph $G$ is the graph obtained from $G$ by adding a new vertex of degree $d$. A \emph{($d$-dimensional) $1$-extension} of $G$ is the graph obtained from $G$ by removing an edge $xy$, and adding a new vertex adjacent to $x$, $y$ and $d-1$ other vertices.
    The inverse of these operations are called \emph{($d$-dimensional) $0$- and $1$-reductions}.

The following lemma is relatively straightforward.

\begin{lemma}[{\cite[Lemma 11.1.1, Theorem 11.1.7]{Wlong}}]
\label{lem:0ext}
Let $G$ be $d$-independent and suppose that $G'$ is obtained from $G$ by a 0-extension or a 1-extension. Then $G'$ is $d$-independent.
\end{lemma}

More difficult is the fact that 1-extensions preserve global $d$-rigidity \cite{Bob05}. 
In fact one can preserve the existence of a PSD equilibrium stress of full rank \cite[Section 9]{Cquestions}. That proof can be adapted to obtain the following lemma which we will need. Details are provided in Appendix \ref{sec:app}.

\begin{lemma}\label{lem:1ext}
    Let $G=(V,E)$ a $d$-rigid graph that has a generic realization $(G,p)$ in $\mathbb{R}^d$ with a PSD equilibrium stress $\omega$ with rank $|V|-d-2$.
	Suppose that $G'=(V',E')$ is a $d$-rigid $d$-circuit obtained from $G$ by a 1-extension.
	Then there exists a generic framework $(G',p')$ with a PSD equilibrium stress of rank at least $|V'|-d-2$.
\end{lemma}

\begin{defn}\label{def: clique sum}
A graph $G$ is a \emph{$k$-sum} of two induced 
subgraphs $G_1$ and $G_2$ each with at least $k+1$ vertices 
if $G$ is the union of $G_1$ and $G_2$ and  $G_1\cap G_2$ is isomorphic to $K_k$.
\end{defn}

Given a graph $G$ with edge $ij$, we let $G-ij$ denote the graph obtained from $G$ by deleting the edge $ij$.

\begin{lemma}[\cite{Betal}]
\label{lem: psd clique sum}
Let $1\le k\le d$ and $G$ be a $k$-sum of subgraphs $G_1$ and $G_2$
and $ij$ and edge of $G_1\cap G_2$.
Suppose that there are generic $d$-dimensional frameworks
$(G_1,p^1)$ and $(G_2,p^2)$ that, respectively, support 
non-zero PSD equilibrium stresses $\omega^1$ and $\omega^2$,
such that $\omega^k_{ij} \neq 0$ for $k = 1,2$.
Let $G' = G- ij$.
Then there is a generic $d$-dimensional framework $(G',p)$
that supports a non-zero PSD equilibrium stress.
\end{lemma}

We will also use the following lemma which was implicit in \cite{Betal}.

\begin{lemma}\label{lem:cone}
Let $G'$ be the cone of $G$. Then $\mlt (G')=\mlt (G)+1$.    
\end{lemma}

\begin{proof}
This is an immediate corollary of \cite[Lemma 4.9]{cgt2}
and Theorem \ref{thm: main mlt stress}.
\end{proof}

Note that iteratively coning $K_{5,5}$ (applying the lemma) gives us a family of non-bipartite graphs $G$ for which $\mlt(G)\neq \gcr(G)$. We also note, but will not need, the fact that adding a single edge to a graph can alter the MLT by at most 1.

\begin{prop}\label{prop:edgeaddition}
    Let $e$ be a non-edge of a graph $G=(V,E)$.
    Then $\mlt (G) \leq \mlt (G+e) \leq \mlt (G) +1$.
\end{prop}

\begin{proof}
    Let $e = vw$ and define the set of non-edges $F := \{ vu : u \in V, ~ v \not\sim u \}$.
    By Lemma \ref{lem:cone},
    $\mlt(G+F) = \mlt (G - v) +1$.
    By Lemma \ref{lem: mlt monotone} we see that
    \begin{align*}
        \mlt(G) \leq \mlt(G+e) \leq \mlt(G + F) = \mlt(G-v) +1 \leq \mlt(G) + 1,
    \end{align*}
    which concludes the proof.
\end{proof}

\section{Graphs with few vertices}\label{sec: many / few}

We will demonstrate that if $G$ has sufficiently few vertices then $\mlt(G) = \gcr(G)$. This requires a number of technical results which we now derive. The first is computational.

\begin{lemma}\label{lemma: 6 regular}
    Every $6$-regular graph on $9$ vertices is globally $4$-rigid.
\end{lemma}

\begin{proof}
    The complement of a $6$-regular graph on $9$ vertices is $2$-regular
    and there are exactly four isomorphism classes of $2$-regular graphs on $9$-vertices.
    In particular, these are the $9$-cycle, the disjoint union of a $6$-cycle and a $3$-cycle, the disjoint union of a $5$-cycle and a $4$-cycle, and the disjoint union of three $3$-cycles.
    
    In all four graphs let the vertex set be $\{v_1,v_2,\dots,v_9\}$. We define $G_1$ by taking the edge set of $K_9$ and deleting the 9-cycle with edges $v_1v_2,v_2v_3,\dots,v_8v_9,v_9v_1$. We define $G_2$ by taking the edge set of $K_9$ and deleting the 6-cycle with edges $v_1v_2,v_2v_3,\dots,v_5v_6,v_6v_1$ and the 3-cycle with edges $v_7v_8,v_8v_9,v_9v_7$. We define $G_3$ by taking the edge set of $K_9$ and deleting the 5-cycle with edges $v_1v_2,v_2v_3,\dots,v_4v_5,v_5v_1$ and the 4-cycle with edges $v_6v_7,v_7v_8,v_8v_9,v_9v_6$. Finally, we define $G_4$ by taking the edge set of $K_9$ and deleting the 3-cycle with edges $v_1v_2,v_2v_3,v_3v_1$, the 3-cycle with edges $v_4v_5,v_5v_6,v_6v_4$ and the 3-cycle with edges $v_7v_8,v_8v_9,v_9v_7$. 

    By~\cite[Theorem~1.3]{Bob05}, it suffices to show that each $G_i$ has an infinitesimally rigid realization with an equilibrium stress of rank 4. For $1\leq i \leq 4$ we define the framework $(G_i,p)$ in $\mathbb{R}^4$ by putting $p(v_1)=(0,0,0,0)$, $p(v_2)=(0,0,0,1)$, $p(v_3)=(0,0,4,-1)$, $p(v_4)=(0,2,3,5)$, $p(v_5)=(1,-1,0,-2)$, $p(v_6)=(1,3,7,0)$, $p(v_7)=(2,-4,-1,1)$, $p(v_8)=(-9,0,2,11)$ and $p(v_9)=(-3,3,1,6)$.
    Given these realizations it is simple for the reader to verify that the rigidity matrix has rank $4n-10=26$, that the cokernel of the rigidity matrix is 1-dimensional and that the stress matrix corresponding to any non-zero equilibrium stress of $(G_i,p)$ has rank 4.
\end{proof}

We will also make use of the following theorem of Jord{\'a}n. Given a graph $G=(V,E)$ and a subset $X$ of vertices,
$i_G(X)$ (or $i(X)$ when the graph $G$ is clear) denotes the number of edges in the subgraph induced by~$X$.
$G$ is \emph{$(d,\binom{d+1}{2})$-sparse} if $i(X)\leq d|X|-\binom{d+1}{2}$ for all $X\subset V$ with $|X|\geq d$. $G$ is \emph{$(d,\binom{d+1}{2})$-tight} if it is $(d,\binom{d+1}{2})$-sparse and $|E|=d|V|-\binom{d+1}{2}$.
A graph is \emph{redundantly $d$-rigid} if it is $d$-rigid, and remains so after removing any edge. 

\begin{thm}[{\cite[Theorems 2.3 and 3.2]{egres-20-01}}]\label{thm:tibor}
   Let $d \ge 1$, let $k \in \{3,4\}$ and let $G$ be a graph on $d+k$ vertices. 
   Then $G$ is $d$-rigid if and only if it contains a spanning $(d,\binom{d+1}{2})$-tight subgraph. Moreover
   $G$ is globally $d$-rigid if and only if $G$ is redundantly $d$-rigid and $(d+1)$-connected.
\end{thm}

The following lemma was implicit in \cite{egres-20-01}. Let $H_d$ denote the graph obtained by gluing two copies of $K_{d+2}$ along a common $K_d$ subgraph and removing a common edge -- see Figure \ref{fig:h3} for an illustration when $d=3$.

\begin{figure}[h]
	\begin{center}
		\begin{tikzpicture}
			\node[vertex] (L1) at (-1,-0.5) {};
			\node[vertex] (L2) at (-1,0.5) {};
			\node[vertex] (C1) at (0,-1) {};
			\node[vertex] (C2) at (0,0) {};
			\node[vertex] (C3) at (0,1) {};
			\node[vertex] (R1) at (1,-0.5) {};
			\node[vertex] (R2) at (1,0.5) {};
	
			\draw[edge] (L1)edge(C1);
			\draw[edge] (L1)edge(C2);
			\draw[edge] (L1)edge(C3);
			
			\draw[edge] (L2)edge(C1);
			\draw[edge] (L2)edge(C2);
			\draw[edge] (L2)edge(C3);
			
			\draw[edge] (L1)edge(L2);
			
			\draw[edge] (R1)edge(C1);
			\draw[edge] (R1)edge(C2);
			\draw[edge] (R1)edge(C3);
			
			\draw[edge] (R2)edge(C1);
			\draw[edge] (R2)edge(C2);
			\draw[edge] (R2)edge(C3);
			
			\draw[edge] (R1)edge(R2);
			
			\draw[edge] (C1)edge(C2);
			\draw[edge] (C2)edge(C3);
		\end{tikzpicture}
	\end{center}
	\caption{The graph $H_3$.}
	\label{fig:h3}
\end{figure}
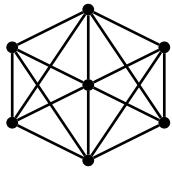

\begin{lemma}[\cite{egres-20-01}]\label{lemma: H_n}
    $H_d$ is the unique graph on at most $d+4$ vertices that is a $d$-rigid $d$-circuit and not $(d+1)$-connected.
\end{lemma}

For disjoint vertex sets $A,B$ of a graph $G$,
we will denote the induced subgraph on the vertex set $A$ by $G[A]$,
the non-edges of $G$ by $E^c$,
the non-edges of $G$ induced on the set $A$ by $E^c[A]$,
and the set of non-edges of $G$ with one end in $A$ and the other in $B$ by $E^c(A,B)$.
The minimum degree of a graph $G$ will be denoted $\delta(G)$.
The set of neighbors of a vertex $v$ of $G$ will be denoted $N_G(v)$.
Given a set $V$ of vertices, $K(V)$ denotes the complete graph on vertex set $V$.
Given a subset $S$ of edges or vertices of $G$, we let $G-S$ denote the graph obtained by removing $S$.
Given graphs $G$ and $H$, $G+H$ denotes the graph whose vertex and edge sets are the unions of the vertex and edge sets of $G$ and $H$.

\begin{lemma}\label{lem:d+5circcut}
    Suppose $G$ is a $d$-rigid $d$-circuit with $d+5$ vertices and minimum degree $d+1$. Then $G$ has a $d$-dimensional generic realization with a PSD equilibrium stress of rank at least~$3$.
\end{lemma}

\begin{proof}
     First suppose $G$ is a $(d+1)$-connected $d$-rigid $d$-circuit with a vertex $v$ of degree $d+1$.
   Then $G-v$ is $d$-rigid.
    Since $G$ is $(d+1)$-connected and redundantly $d$-rigid,
    the graph $G'=G-v_0+K(N_G(v_0))$ must also be $(d+1)$-connected and redundantly $d$-rigid.
    By Theorem~\ref{thm:tibor},
    $G'$ is globally $d$-rigid,
    and hence by \cite[Lemma 4.1]{t1},
    $G$ is globally $d$-rigid.
    The result follows from Theorem \ref{thm: gen ur}.
    
    Now suppose $G$ is not $(d+1)$-connected.
    Since $G$ is $d$-rigid, there exists a separating set $C \subset V$ of size $d$.
    As $\delta(G)=d+1$,
    $G-C$ will have exactly two connected components $G[A],G[B]$ where $A=\{a_1,a_2\}$ and $B=\{b_1,b_2,b_3\}$. The complement $G^c$ of $G$ has exactly 9 edges. Since $K_{d+2}$ is not a subgraph of $G$, $G[C]$ is not complete. Since $|E^c(A, B)|=6$ it follows that 
  $|E^c[B \cup C]|=3$ and $|E^c[C]| \in \{1,2,3\}$. 
    
    If $|E^c[C]|=1$ and $G[B]\cong K_3$ then, without loss of generality, $b_1,b_2$ have degree $d+1$ in $G$ and are incident to edges of $E^C(B,C)$. We can apply a 1-reduction at $b_1$ and add the missing edge incident to $b_2$ to result in a smaller $d$-circuit. A similar argument applies if $G[B]$ has 2 edges. In both cases the resulting $d$-circuit has $d+4$ vertices and is not $(d+1)$-connected. Hence it is $H_d$.
    This case is completed by Lemma \ref{lem:1ext}.
    
    Now assume $|E^c[C]| \geq 2$.
    If three non-edges meet at a vertex $c \in C$,
    then $G[B \cup C -c]$ is isomorphic to $K_{d+2}$,
    contradicting that $G$ is a $d$-circuit.
    Hence there are not,
    and there exists a 1-reduction at $a_1$ followed by a 0-reduction at $a_2$ resulting in the graph $K_{d+3}-\{e,f\}$ where $e$ and $f$ do not share a vertex. The result now follows from Theorem~\ref{thm:tibor}, Theorem~\ref{thm: gen ur}, and Lemma~\ref{lem:1ext}.
\end{proof}

The proof of the following lemma is long and technical, so we defer it to the end of the paper (see Section \ref{app:b}).

\begin{lemma}\label{lem:3d9circcut}
    Suppose $G$ is a 3-rigid 3-circuit with $9$ vertices.
    Then $\mlt(G) \geq 5$.
\end{lemma}

\begin{lemma}\label{lemma: d circuit d+4 vertices}
    Let $G$ be a $d$-circuit on at most $d+4$ vertices.
    Then $\mlt(G) \ge d+2$.
\end{lemma}

\begin{proof}
    By Lemma~\ref{lemma: H_n}, if $G$ is not $(d+1)$-connected, then $G=H_d$.
    Since $K_{d+2}$ has a PSD equilibrium stress it follows from Lemma \ref{lem: psd clique sum} that $H_d$ has a PSD equilibrium stress and hence Theorem 
    \ref{thm: main mlt stress} implies that $\mlt(H_d)\ge d+2$.
    If $G$ is $(d+1)$-connected, then $G$ is globally rigid by Theorem~\ref{thm:tibor} and Corollary \ref{cor:gr+cir} gives the result.
\end{proof}

\begin{proof}[Proof of Theorem \ref{THM:SMALL}]
For any graph $H$, if we choose $d$ so that $\gcr(H)=d+2$ then $H$ is not $d$-independent and therefore contains a $d$-circuit.
In light of Lemma~\ref{lem: mlt monotone}, it therefore suffices to prove any $d$-circuit $G$ with $9$ or fewer vertices has $\mlt(G) \geq d+2$.
By Theorem~\ref{thm: equality of gcr and mlt}, we may assume $d \ge 3$ and by Lemma~\ref{lemma: d circuit d+4 vertices}, we may assume $n \ge d+5$.
Thus either $d=3$ and $n\in \{8,9\}$, or $d=4$ and $n=9$. 

Assume that $G$ is $d$-rigid, i.e.~that it has $dn-\binom{d+1}{2}+1$ edges.
If $d=3$ and $n=9$ then $\mlt(G) \geq d+2$ by Lemma \ref{lem:3d9circcut}.
If $G$ has a vertex of degree $d+1$ and $d = 3$ and $n = 8$, or $d=4$ and $n=9$,
then $\mlt(G) \geq d+2$ by Lemma \ref{lem:d+5circcut} and Theorem 
\ref{thm: main mlt stress}.
If $G$ does not have a vertex of degree $d+1$, then $d=4$, $n=9$, and $G$ is $6$-regular.
In this case, Lemma~\ref{lemma: 6 regular} implies $G$ is globally rigid and
Corollary~\ref{cor:gr+cir} gives the result.

Thus we may assume that $G$ is not $d$-rigid.
Since $n\leq 9$ and $d\geq 3$, \cite[Theorem 1]{grasegger2020flexible} implies that $G$ is obtained from two rigid $d$-circuits by gluing them together over a common complete subgraph on $(d-1)$ or $(d-2)$ vertices and deleting exactly one edge from the intersection.  From above, the two rigid $d$-circuits that we glue
are globally $d$-rigid.  Theorem \ref{thm: gen ur} implies that there is a generic framework for each of these with 
a non-zero PSD equilibrium stress.  We may then apply  Lemma \ref{lem: psd clique sum} and Theorem \ref{thm: main mlt stress}
to conclude that $\mlt(G) = d+2$.
\end{proof}

\section{Graphs with few edges}
\label{sec:few edges}

The purpose of this section is to prove Theorem \ref{thm:atmost24}.
We will use the following results in the proof.

\begin{thm}[Jackson and Jord\'an \cite{JJ}]\label{thm:jj3sparse}
Let $G=(V,E)$ be a graph and suppose $i(X)\leq \frac{1}{2}(5|X|-7)$ for all $X\subset V$ with $|X|\geq 2$ then $G$ is 3-independent.
\end{thm}

The next lemma generalises \cite[Lemma 12]{grasegger2020flexible} but a careful reading of their proof shows that it works in the generality we state.

\begin{lemma}\label{lem:1ext of 2sum}
    Let $G$ be the 2-sum of two 3-rigid 3-circuits $G_1,G_2$.
    If $G'$ is formed from $G$ by a 1-extension,
    then $G'$ is either minimally 3-rigid,
    or $G'$ is the 2-sum of two 3-rigid 3-circuits.
\end{lemma}

\begin{lemma}[{\cite[Lemma 11.1.9]{Wlong}}]
\label{lem:rankplus1}
Let $G_1$, $G_2$ be subgraphs of a graph $G$ and suppose that $G = G_1 \cup G_2$.
\begin{enumerate}[(a)]
    \item If $G_1 \cap G_2$ is 3-rigid and $G_1$,$G_2$ are 3-independent then $G$ is 3-independent.
    \item If $|V (G_1) \cap V (G_2)|\leq 2$, $u\in V(G_1) \setminus V(G_2)$ and $v_2 \in V(G_2) \setminus V (G_1)$ then $\rank R(G + uv,p) = \rank R(G,p) + 1$ for any generic realization $p$ in $\mathbb{R}^3$. 
\end{enumerate} 
\end{lemma}

In the proof of Theorem \ref{thm:atmost24} we will also use repeatedly the following counting argument. First we need some definitions. Let $G=(V,E)$ be a graph. 
Take $X,Y\subset V$. We will say that $X$ is \emph{3-critical} if $|X|\geq 3$ and $i(X)=3|X|-6$. Also, we will use $d_G(X,Y)$ to denote the number of edges in $G$ of the form $xy$ where $x\in X\setminus Y$ and $y\in Y\setminus X$ (again we will simply use $d(X,Y)$ if the graph is clear from the context).

\begin{rmk}\label{rem:count}
Let $G=(V,E)$ be $(3,6)$-sparse and suppose that $X,Y\subset V$ are 3-critical. Then
\begin{align*}
    3|X|-6+3|Y|-6 = i(X)+i(Y) &=i(X\cup Y)+i(X\cap Y)-d(X,Y) \\
    &\leq 3|X\cup Y|-6+i(X\cap Y)-d(X,Y).
\end{align*}
If $|X\cap Y|=2$ and $G[X\cup Y]$ has no edges or $|X\cap Y|\geq 3$ then $i(X\cap Y)=3|X\cap Y|-6$ and hence equality holds throughout and $d(X,Y)=0$. If $|X\cap Y|=2$ and $G[X\cap Y]\cong K_2$ then
$i(X\cap Y)=1$, $d(X,Y)\leq 1$ and $i(X\cup Y)=3|X\cup Y|-6-(1-d(X,Y))$. Lastly if $|X\cap Y|=1$ then $i(X\cap Y)=0$, $d(X,Y)\leq 3$ and $i(X\cup Y)=3|X\cup Y|-6-(3-d(X,Y))$.
\end{rmk}

\begin{lemma}\label{lem:4between}
    Let $G$ be a 3-connected graph formed from two disjoint 3-independent graphs $G_1$ and $G_2$ by adding at most 4 edges between them.
    Then $G$ is 3-independent.
\end{lemma}

\begin{proof}
Let $F$ be the set of at most 4 edges between $G_1$ and $G_2$. Suppose that $F$ is an independent (in the graph theoretical sense) set in $G$. Then the result follows from an elementary body-bar type argument (see, for example, \cite{TAY198495}).
Since $G$ is 3-connected in each remaining possibility there exists an edge $e\in F$ such that $G-e$ has a 2-vertex-separation $\{x,y\}$ where $x,y\in G_1$ (by relabelling $G_1,G_2$ if necessary) and $x,y$ are incident to edges of $F$. Since $G_1$ and $G_2\cup \{x,y\}$ are 3-independent (by Lemma \ref{lem:0ext}) we may now use Lemma \ref{lem:rankplus1}(b) to see that $G$ is 3-independent.
\end{proof}

\begin{proof}[Proof of Theorem \ref{thm:atmost24}]
    Choose any $G=(V,E)$ with $|E| \leq 24$.
    If $|V| \leq 9$ then $\mlt (G) =  \gcr (G)$ by Theorem \ref{THM:SMALL},
    so we may assume that $|V| \geq 10$.
    Similarly, if $\gcr (G) \leq 4$ then $\mlt (G) =  \gcr (G)$ by Theorem \ref{thm: equality of gcr and mlt},
    so we may suppose that $\gcr (G) \geq 5$.
    Fix $d = \gcr(G) - 2 \geq 3$.
    Since $G$ is $d$-dependent it contains a $d$-circuit $H$ and since $\gcr(G)=d+2$ we have $\gcr(H)=d+2$. Now $\mlt (H)\leq \mlt (G)\leq \gcr (G)=d+2$ (by Theorem \ref{thm: uhler gcr mlt}).
    Hence it will suffice for us to prove that if $G$ is a $d$-circuit then $\mlt(G) = \gcr(G)$.
    
    As $G$ is a $d$-circuit it is easy to deduce from Lemma \ref{lem:0ext} that $\delta(G) \geq d+1$.
    Since $|V|\geq 10$, the handshaking lemma implies that 
    \begin{eqnarray}\label{eqn:deg}
    24 \geq |E| \geq \frac{d+1}{2}|V| \geq 5d + 5.\end{eqnarray}
    Hence $d = 3$ and now Equation (\ref{eqn:deg}) implies $24\geq 2|V|$ with equality if and only if $G$ is 4-regular.
    If $\delta (G) \geq 5$ then $|E| \geq \frac{5}{2}|V|$ and, since $|V| \geq 10$, this contradicts the hypothesis that $|E|\leq 24$.
    Hence $\delta(G) = 4$.
    If $\Delta(G) = |V| -1$ then $G$ has a cone vertex $u$, and $\gcr(G - u) = \gcr(G) - 1 = 4$ (since coning takes a $d$-independent graph to a $(d+1)$-independent graph \cite{W83}).
    By Theorem \ref{thm: equality of gcr and mlt}, it follows that $\mlt(G-u) = \gcr(G-u) = 4$.
    By Lemma \ref{lem:cone}, we have 
    $$\mlt(G) = \mlt(G-u) + 1 = \gcr(G-u) + 1 = \gcr(G).$$
    Hence we may suppose that $\Delta (G) \leq |V| - 2$.
    If $\Delta(G) \leq 5$ then $\mlt(G) = \gcr(G)$ by Theorem \ref{thm:degbounded},
    hence we may suppose that $\Delta(G) \geq 6$ and so $G$ is not 4-regular. Thus
    $|V| \in \{10,11\}$.
    
    If $G$ is not 3-connected (as $d$-circuits must be 2-connected) then it is the 2-sum of two 3-circuits $G_1,G_2$, each with at most 7 vertices. 
    By Theorem \ref{THM:SMALL} and Corollary \ref{cor:gr+cir},
    $\mlt(G_1) = \mlt(G_2) = 5$.
    Hence by Theorem \ref{thm: main mlt stress} and Lemma \ref{lem: psd clique sum},
    $\mlt(G) \geq 5$.
    Since $\gcr(G) = 5$,
    $\mlt(G) = 5$ by Theorem \ref{thm: equality of gcr and mlt}.
    So we may assume that $G$ is 3-connected.
    
    \textbf{Case 1: $|V|=11$.}
    For $X \subset V$ with $|X|\geq 2$,
    define the value $f(X) := \frac{1}{2}(5|X| - 7)$.
    For $X \subset V$ with $2 \leq |X| \leq 4$ we have 
    $i_G(X) \leq \binom{|X|}{2} \leq f(X)$,
    and $i_G(V) \leq 24 = f(V)$.
    Since $G$ is a 3-circuit, it does not contain a copy of $K_5$,
    hence $i_G(X) \leq 9 \leq f(X)$ whenever $|X|=5$.
    When removing $k\leq 4$ vertices,
    we must remove at least $\sum_{i=1}^k 4 - (i-1)$ edges (the minimum is achieved by removing a degree 4 vertex and then removing all but one of its neighbours),
    hence $i_G(X) \leq f(X)$ if $|X| \geq 7$.
    If $|X| = 6$,
    then $i_G(X)\leq 3|X|-6=12$ 
    and
    $f(X)=\frac{1}{2}(5|X|-7)=\frac{23}{2}$. Hence $i_G(X)>f(X)$ if and only if $G[X]$ is minimally 3-rigid.
    Suppose that this is the case.
    As $G$ is a 3-circuit,
    $G[V \setminus X]$ is 3-independent.
    Let $t$ denote the number of edges between $X$ and $V\setminus X$. Then, since $\delta(G)\geq 4$ and $i_G(V \setminus X) \leq 12 -t$, we have
    $$ 20=4|V\setminus X| \leq \sum_{v\in V\setminus X} d_G(v)=2i_G(V\setminus X)+t\leq 2(12-t)+t=24-t,$$
    and hence $t\leq 4$. Now $G$ is 3-independent by Lemma \ref{lem:4between}, a contradiction. 
    Hence $i_G(X) \leq f(X)$ if $|X| = 6$ and so $i_G(X) \leq f(X)$ for all $X\subset V$ with $|X|\geq 2$. Thus
    $G$ is 3-independent by Theorem \ref{thm:jj3sparse},
    a contradiction.
    
    \textbf{Case 2: $|V|=10$.}
    As $\delta(G) =4$ and $\Delta (G) \geq 6$,
    $|E| \geq 21$.
    If $|E|=21$ then $i_G(X) \leq f(X)$ for all $X \subset V$. This follows similarly to the above. In particular it is clear for all $X\subset V$ with $2\leq |X|\leq 4$, it holds for $|X|=5$ since $K_5$ is a 3-circuit, and for all $6\leq |X|\leq 10$ it follows from repeated removal of minimum degree vertices.
    Hence $|E| \in \{22,23,24\}$.
    Since $|E| \leq 24 < 3|V| - 5$,
    $G$ is 3-flexible.
    Furthermore,
    since every proper subgraph of $G$ is 3-independent,
    $G$ is $(3,6)$-sparse.
     Let $u$ be a degree 4 vertex.
    
    \begin{claim}
         There exists a pair $x,y\in N(u)$ such that $H:= G-u +xy$ is $(3,6)$-sparse.
    \end{claim}
    
    \begin{proof}
  Let $N(u)=\{x,y,z,w\}$. Since $G$ does not contain a subgraph isomorphic to $K_5$, without loss of generality we may suppose that $xy\notin E$. Now, by \cite[Lemma 3.1]{JJ}, $G-v+xy$ is not $(3,6)$-sparse if and only if there exists a 3-critical set $X\subset V-u$ with $x,y\in X$. Suppose $X$ is the maximal such set with respect to inclusion. If $N(v)\subset X$ then we contradict the $(3,6)$-sparsity of $G$ so without loss of generality we may assume $w\notin X$. 
  
  Consider the pair $\{y,w\}\subset N(v)$. If there exists a 3-critical set $W\subset V-u$ with $y,w\in W$ then by Remark \ref{rem:count} either $i(X\cap W)\leq 3|X\cap W|-6$, or $|X\cap W|=1$, or $|X\cap W|=2$ and $G[X\cap W|\cong K_2$. In the first case $i(X\cup W)=3|X\cup W|-6$ contradicting the maximality of $X$. In the second case we have $V=X\cup W\cup \{u\}$ so we may suppose $z\in X$.
  We now note that both $G[X \cup \{u\}]$ and $G[W]$ are $(3,6)$-tight, and hence minimally 3-rigid as $G$ is a 3-circuit.
  Hence the graph $G'$ formed from gluing $G[X \cup \{u\}]$ and $G[W \cup \{u\}]$ at the edge $uy$ is 3-independent with 1 degree of freedom by Lemma \ref{lem:rankplus1}(a).
  If $d(X,W) =0$,
  then $G=G'$, contradicting the hypothesis that $G$ is a 3-circuit.
  If $d(X,W) \geq 1$ then, by Lemma \ref{lem:rankplus1}(b),
  $G$ is 3-rigid,
  contradicting the fact that $G$ is 3-flexible. Hence we have $|X\cap W|=2$ and $G[X\cap W|\cong K_2$ (note that $x\notin W$). If $V=X\cup W\cup u$ then without loss of generality we may assume $z\in X$. Now, since $d(X,W)=0$, the 3-independent graph $G-uw$ has a 2-vertex-separation $X\cap W:=\{y,y'\}$ where $yy'\in E$. It follows from Lemma \ref{lem:rankplus1}(b) that $G$ is 3-rigid, a contradiction. So there exists a vertex in $V\setminus (X\cup W\cup u)$. Since $|E|\leq 24$ and $\delta(G)=4$, this vertex must be $z$ and $d(z)=4$. Note that $xw\notin E$ and $G-\{u,z\}$ is 3-independent, so $G-\{u,z\}+xw$ is 3-independent by Lemma \ref{lem:rankplus1}(b). Hence $G-u+xw$ is 3-independent by Lemma \ref{lem:0ext} giving the claim. Hence we may assume that $yw\in E$ and by symmetry that $xw\in E$. 
  
  Suppose that $z\in X$. If $zw\in E$ then we contradict the maximality of $X$, so we may suppose $zw\notin E$. Then there exists a 3-critical subset $Y\subset V-u$ with $w,z\in Y$. By the maximality of $X$ we have $|X\cap Y| \leq 2$. If $|X\cap Y|=1$ then, since $|E|\leq 24$, $\{w,z\}$ is a 2-vertex-separation of $G$ contradicting the fact that $G$ is 3-connected. So $|X\cap Y|= 2$ but then we contradict the hypothesis that $|E|\leq 24$.
  
  Finally, suppose that $z\notin X$. Then as above we deduce that $xz,yz\in E$. Suppose that $wz\notin E$ and there exists a 3-critical set $Z\subset V-u$ with $w,z\in Z$. If $x$ or $y$ is contained in $Z$ then we may relabel and apply the above argument to obtain a contradiction. Hence we may assume $x,y\notin Z$. Since $|V|=10$ and $X$ is maximal we have $1\leq |X\cap Z|\leq 2$. Now Remark \ref{rem:count} implies that $d(X,Z)\leq 3$ but we have already shown that $xw,xz,yw,yz\in E$. This is a contradiction and so we must have $wz\in E$. Now $5\leq |X|\leq 7$. If $|X|\in \{5,6\}$ then since $G[X\cup \{u,w,z\}]$ is $(3,6)$-tight and $|E|\leq 24$ there must exist a vertex in $V\setminus (X\cup \{u,w,z\})$ of degree at most 3, contradicting the fact that $\delta(G)=4$.
  Finally, if $|X|=7$ then, since $|E|\leq 24$, $\{x,y\}$ is a 2-vertex-separation of $G$. This contradiction completes the proof of the claim. 
    \end{proof}
    
    If $H$ is 3-independent then $G$ is also 3-independent as $G$ is formed from $H$ by a 1-extension,
    hence $H$ is 3-dependent.
    Thus there exists a set $X \subset V-u$ such that $H[X]$ is a 3-circuit.
    If $\{x,y\} \not\subset X$ then $H[X] \subset G$,
    contradicting that $G$ is a 3-circuit,
    hence $x,y\in X$.
    As $H[X]$ is a $(3,6)$-sparse 3-circuit with at most 9 vertices,
    $H[X]$ is the 2-sum of two 3-circuits $H_1,H_2$ where $|V(H_1)| = 5$ and $|V(H_2)| \in \{5,6\}$.
    By Theorem \ref{thm:tibor}
$H_1$ and $H_2$ are 3-rigid.
    By Lemma \ref{lem:1ext of 2sum},
    either $G$ is the 2-sum of two 3-circuits $G_1,G_2$ (and hence not 3-connected), or $G$ is 3-rigid.
    However, both cases contradict our assumptions.
\end{proof}

\section{Many edges and bounded degrees}
\label{sec:manybound}

In this short section we will demonstrate that if $G$ is sufficiently close to being complete, or has sufficiently small vertex degrees then $\mlt(G) = \gcr(G)$.

\begin{proof}[Proof of Theorem \ref{t:manyedges}]
 Let $G$ be a graph on $n$ vertices and fix $d = \gcr(G)-2$. If $n\leq 4$ the theorem is trivial (for example, from Theorem \ref{thm: equality of gcr and mlt}) so suppose $n\geq 5$.
The graph $G$ has at least $\binom{n}{2} -5$ edges. We have
\begin{align*}
    \binom{n}{2} -5 > (n-4)n - \binom{n-3}{2}.
\end{align*}
Since any $(n-4)$-independent graph is $(n-4,\binom{n-3}{2})$-sparse, $G$ is $(n-4)$-dependent and therefore $G$ contains an $(n-4)$-circuit and $d \geq n-4$. It follows from the definition of the GCR that $G$ contains a $d$-circuit on at most $d+4$ vertices.
Lemmas~\ref{lem: mlt monotone} and~\ref{lemma: d circuit d+4 vertices} and Theorem \ref{thm: uhler gcr mlt} now imply $\mlt(G) = d+2$, completing the proof.
\end{proof}

The next proof requires the following result of Jackson and Jord\'{a}n.

\begin{thm}[\cite{JJ}]\label{thm:jacksonjordan}
    Let $G$ be a connected graph with minimum degree at most $d+1$ and maximum degree at most $d+2$.
    Then $G$ is $d$-independent if and only if $i(X) \leq d |X| - \binom{d+1}{2}$ for any vertex set $X \subset V$ with $|X| \geq d+2$.
\end{thm}

\begin{proof}[Proof of Theorem \ref{thm:degbounded}]
    The degree hypothesis implies that any set $X \subset V$ satisfies $i(X) \leq \frac{1}{2}(5|X| - 1)$.
    Hence $i(X) \leq 4|X| -10$ for all $|X|\geq 6$ and $i(X) \leq 3|X| -6$ for all $|X| \geq 10$\footnote{The cases when $|X|=6$ and respectively $|X|=10$ follow since $i(X)$ is an integer.}. Applying these two observations with Theorem \ref{thm:jacksonjordan} implies that we have $\gcr(G) \leq 5$,
    with equality if and only if $G$ contains a 3-circuit on at most 9 vertices (since no 4-circuit exists on at most 5 vertices).
   If $\gcr(G)\leq 4$ then Theorem \ref{thm: equality of gcr and mlt} gives the result. Hence we may
    suppose $\gcr(G) =5$ and let $H$ be a 3-circuit contained in $G$ with $|V(H)| \leq 9$.
    The result follows from applying Theorem \ref{THM:SMALL} to $H$.
\end{proof}

\section{Completing the proof of Theorem~\ref{THM:SMALL}}
\label{app:b}

It remains to prove Lemma \ref{lem:3d9circcut}.
We first deal with the case when $G$ is 4-connected.
In what follows, we make repeated implicit use of the fact that every vertex in a $d$-circuit has degree at least $d+1$.

\begin{lemma}\label{lem:3d9gr}
    Let $G$ be a 4-connected 3-rigid 3-circuit on 9 vertices.
    Then $G$ has a generic realization in $\mathbb{R}^3$ with a PSD equilibrium stress.
\end{lemma}

\begin{proof}
    A counting argument shows that $G$ has a vertex $v_0$ of degree 4.
    Since $G$ is a 3-circuit,
    there exist distinct vertices $x,y$ adjacent to $v_0$ such that $xy \notin E$. Let $G'$ be the result of the 1-reduction at $v_0$ that adds $xy$. Then $G'$ contains a 3-circuit $H$.
    
    If $|V(H)|=8$ then $H = G'$.
    The connectivity of $H$ is at least $3$ (as otherwise $G$ would not be 4-connected),
    hence by \cite{grasegger2020flexible} $H$ is 3-rigid.
    The result now follows from Lemmas \ref{lem:1ext} and  \ref{lem:d+5circcut}.
    
    If $|V(H)|=7$ then $G'$ is formed from $H$ by a 0-extension that adds a vertex $v_1$.
    Since $G$ is a 4-connected 3-circuit,
    $v_0$ and $v_1$ must be adjacent in $G$,
    and $v_1 \notin \{x,y\}$.
    We now note that $G$ can be formed from $H$ by two 1-extensions;
    the first will remove the edge $xy$ and connect the vertex $v_0$ to $N_G(v_0) -v_1 + u$ for some vertex $u \in V(H)$,
    and the second will remove the edge $uv_0$ and attach $v_1$ to all its neighbours in $G$. 
    Since $H_3$ has a 3-dimensional generic realization with a PSD equilibrium stress of rank 2 and any globally 3-rigid graph has a PSD equilibrium stress of rank 3 (Theorem \ref{thm: gen ur}),
    the result now follows from Lemmas~\ref{lem:1ext} and~\ref{lemma: H_n} and Theorem~\ref{thm:tibor}.
    
    If $|V(H)|=6$ then $H$ is globally 3-rigid by Theorem~\ref{thm:tibor}.
    Since $|E(G')|= 19$ and $|E(H)|=13$,
    $G'$ has 6 edges not in $H$.
    Given $a,b$ are the two vertices in $G' -V(H)$ with $a$ having equal or higher degree than $b$ in $G'$,
    one of the two possibilities must hold: (i) both $a$ and $b$ have degree 3 in $G'$,
    or (ii) $a$ has degree 4 in $G'$, $b$ has degree 3 in $G'$, and there exists an edge between $a$ and $b$.
    In both cases we have that $v_0 b \in E$, and in (i) we have $v_0a \in E$ also.
    In case (i) we must have distinct vertices $s,t \in V \setminus \{a,b,v_0,x,y\}$ adjacent to $a$ and $b$ respectively as otherwise $G$ would not be 4-connected. Hence in case (i) we can obtain $G$ from $H$ by three 1-extensions;
    the first to add $v_0$ attached to $x,y,s,t$ and the next to split two of the edges $v_0 s$ and $v_0 t$ and add the vertices $a,b$.
    If case (ii) holds then $G$ can be formed from $H$ by a 0-extension to add $b$ adjacent to its neighbours plus a vertex $w$ in the neighbourhood of $a$ in $G$,
    a 1-extension at $w b$ to add $a$,
    and a 1-extension at $xy$ to add $v_0$.
    In either case,
    the result will hold by Lemma \ref{lem:1ext} and Theorem \ref{thm: gen ur}.
    
    Finally,
    suppose $|V(H)|=5$,
    i.e.~$G'$ has a 5-clique. 
    Let $a,b,c$ be the three vertices in $G' -V(H)$.
    We will show that there exists a 1-reduction of $G$ at $a,b$ or $c$ resulting in a graph that does not contain a 5-clique,
    hence reducing the problem to one of the previous cases.
    We first note that $v_0$ can be adjacent to at most two of $a,b,c$ as $x,y \in V(H)$.
    If any of $a,b,c$ are adjacent to four vertices in $H$ then $G$ must contain either $K_6-\{e,f\}$ ($e,f$ independent) or $K_5$ which contradicts that $G$ is a 3-circuit.
    The 4-connectivity of $G$ implies that each of $a,b,c$ has a neighbour in $H$. If $a$ has exactly 1 neighbour in $H$ then $a$ is adjacent to $v_0$ and has degree 4. By a quick case analysis we can see that there is a 1-reduction at $a$ (in $G$) creating a graph with no 5-clique. 
    Hence we may assume each of $a,b,c$ has either 2 or 3 neighbours in $H$.
    If all three have 3 neighbours in $H$ then $G$ would have a vertex of degree 3,
    hence we may assume $a$ has only two neighbours in $H$. If $a,b,c$ all have two neighbours in $H$ then we may assume that $av_0\notin E$ and hence $a$ has degree 4. As above, we can apply a 1-reduction at $a$ (in $G$) to create a graph with no 5-clique. Hence we may assume that $b$ has 3 neighbours in $H$. 
    If $c$ has 2 neighbours in $H$, then a quick case analysis shows that one of $a,b,c$ has degree 4 in $G$ and we again reduce that vertex instead of $v_0$. Lastly if $c$ has 3 neighbours in $H$, then $a$ certainly has degree 4 in $G$ (otherwise $G'$ would have too many edges) and we finish in the same manner. 
\end{proof}

We lastly deal with the case when $G$ is not 4-connected. It will be convenient to define a \emph{node} of $G$ to be a vertex of degree 4 and to use $N$ to denote the set of nodes of $G$.
We also define a \emph{deleted $k$-sum} of two graphs $G_1,G_2$ to be the graph obtained by gluing $G_1$ and $G_2$ along a common $k$-clique, then removing one edge from this common clique.

\begin{lemma}
    Let $G$ be a 3-rigid 3-circuit on 9 vertices with a separating set $C= \{c_1,c_2,c_3\}$.
    Then $G$ has a generic realization in $\mathbb{R}^3$ with a PSD equilibrium stress.
\end{lemma}

\begin{proof}
Let $A,B \subset V$ be chosen such that $|A|\leq|B|$, $A \cup B \cup C = V$,
$A \cap B= A \cap C = B \cap C = \emptyset$,
and there exist no edge joining $A$ and $B$.
As $G$ is a 3-circuit on 9 vertices,
either $|A|=|B|=3$, or $|A|=2$ and $|B|=4$.
Note that $C$ cannot be a clique, as this would imply either $G[A \cup C]$ or $G[B \cup C]$ is a 3-circuit.

Suppose that $G[B]$ is disconnected. Then $|B|=4$ and so $|A|=2$.
	Since 3-circuits have minimum degree at least 4, the only 3-circuit that satisfies the above conditions is the graph described in Figure \ref{fig:3k5s}(a).
	This graph can be formed by the union of three copies of $K_5$ glued at three vertices $\{c_1,c_2,c_3\}$ with two edges $c_1c_2,c_1c_3$ removed.
	It can be shown that $G$ has a 3-dimensional generic realization with a PSD equilibrium stress\footnote{Since equilibrium stresses are invariant under affine transformations, we can find three generic frameworks $(K_5,p^1),(K_5,p^2),(K_5,p^3)$ with PSD equilibrium stresses $\omega^1,\omega^2,\omega^3$ respectively so that: (i) $p^1_{c_i}=p^2_{c_i}=p^3_{c_i}$ for each $i \in \{1,2,3\}$, (ii) $\omega^1_{c_1 c_2} + \omega^2_{c_1 c_2} + \omega^3_{c_1 c_2}=0$,
	(iii) $\omega^1_{c_1 c_3} + \omega^2_{c_1 c_3} + \omega^3_{c_1 c_3}=0$,
	and (iv) the framework $(G,p)$ formed by gluing all three frameworks at the vertices $c_1,c_2,c_3$ and deleting the edges $c_1c_2,c_1c_3$ is regular. The obtained framework will have a PSD equilibrium stress $\omega = \overline{\omega}^1 + \overline{\omega}^2 + \overline{\omega}^3$,
	where $\overline{\omega}^i$ is the extension of $\omega^i$ to the edges of $G + c_1c_2 + c_1 c_3$.}. Hence we may suppose both $G[A]$ and $G[B]$ are connected. 
 	
\begin{figure}[ht]
	\begin{center}
		\begin{tikzpicture}[scale=0.7]
			\node[vertex] (a1) at (0.866,-2.5) {};
			\node[vertex] (a2) at (-0.866,-2.5) {};
			
			\node[vertex] (b1) at (1.732,2) {};
			\node[vertex] (b2) at (2.6,0.5) {};
			
			\node[vertex] (c1) at (-1.732,2) {};
			\node[vertex] (c2) at (-2.6,0.5) {};
			
			\node[vertex] (m1) at (0,1) {};			
			\node[vertex] (m2) at (0.866,-0.5) {};
			\node[vertex] (m3) at (-0.866,-0.5) {};
	
			\draw[edge] (m1)edge(a1);
			\draw[edge] (m1)edge(a2);			
			\draw[edge] (m2)edge(a1);
			\draw[edge] (m2)edge(a2);			
			\draw[edge] (m3)edge(a1);
			\draw[edge] (m3)edge(a2);
			
			\draw[edge] (m1)edge(b1);
			\draw[edge] (m1)edge(b2);			
			\draw[edge] (m2)edge(b1);
			\draw[edge] (m2)edge(b2);			
			\draw[edge] (m3)edge(b1);
			\draw[edge] (m3)edge(b2);
	
			\draw[edge] (m1)edge(c1);
			\draw[edge] (m1)edge(c2);			
			\draw[edge] (m2)edge(c1);
			\draw[edge] (m2)edge(c2);		
			\draw[edge] (m3)edge(c1);
			\draw[edge] (m3)edge(c2);
			
			\draw[edge] (a1)edge(a2);
			\draw[edge] (b1)edge(b2);
			\draw[edge] (c1)edge(c2);
			
			\draw[edge] (m3)edge(m2);
			
\node [rectangle,draw=white, fill=white] (a) at (0,-4) {(a)};
		\end{tikzpicture}
\hspace{1.5cm}
\begin{tikzpicture}[very thick,scale=1]
	\node[vertex] (L1) at (-1,-0.5) {};
			\node[vertex] (L2) at (-1,0.5) {};
			\node[vertex] (C1) at (0,-1) {};
			\node[vertex] (C2) at (0,0) {};
			\node[vertex] (C3) at (0,1) {};
			\node[vertex] (R1) at (1,-1) {};
			\node[vertex] (R2) at (1,0) {};
			\node[vertex] (R3) at (1,1) {};
	
			\draw[edge] (L1)edge(C1);
			\draw[edge] (L1)edge(C2);
			\draw[edge] (L1)edge(C3);
			
			\draw[edge] (L2)edge(C1);
			\draw[edge] (L2)edge(C2);
			\draw[edge] (L2)edge(C3);

			\draw[edge] (L1)edge(L2);

			\draw[edge] (R1)edge(C1);
			\draw[edge] (R1)edge(C2);
			\draw[edge] (R1)edge(C3);
			
			\draw[edge] (R2)edge(C1);
			\draw[edge] (R2)edge(C2);
			\draw[edge] (R2)edge(C3);
			
			\draw[edge] (R3)edge(C1);
			\draw[edge] (R3)edge(C2);
			\draw[edge] (R3)edge(C3);
			
			\draw[edge] (R1)edge(R2);
			\draw[edge] (R2)edge(R3);
			\path[edge] (R1) edge [bend right] node {} (R3);	
			
\node [rectangle,draw=white, fill=white] (b) at (0,-2.5) {(b)};
	\end{tikzpicture}
	\end{center}
	\caption{(a) A graph formed from gluing three copies of $K_5$ at three vertices and then deleting two edges in their intersection. (b) A 3-rigid 3-circuit.}
	\label{fig:3k5s}
\end{figure}
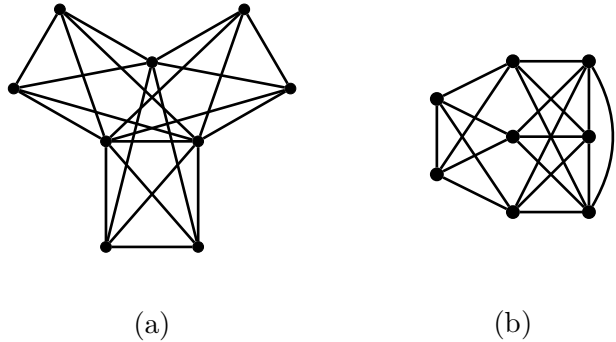

\begin{claim}
	If $A=\{a_1,a_2,a_3\}$ and $B=\{b_1,b_2,b_3\}$,
	then $G$ has a 3-dimensional generic realization with a PSD equilibrium stress.
\end{claim}

\begin{proof}
	Note that $|E^c| = 14$ and $|E^c(A,B)| = 9$.
	If $|E^c[C]|=1$ then $G$ is a deleted 3-sum of two smaller graphs.
	The result now follows from \cite[Lemma 17(a)]{grasegger2020flexible}, Lemma \ref{lem: psd clique sum} and the fact that all 3-circuits on 7 or fewer vertices support a PSD equilibrium stress (see Lemma~\ref{lemma: d circuit d+4 vertices}).
	
	Now suppose $|E^c[C]| \in \{2,3\}$.
	Since $|A|=|B|$, we may assume, without loss of generality, that the number of non-edges with an end in $A$ is less than the number of non-edges with an end in $B$;
	we shall define these sets as $E^c_A := E^c[A \cup C] \setminus E^c[C]$ and $E^c_B := E^c[B \cup C] \setminus E^c[C]$.
	We now have three possible cases;
	$|E^c_A|=0$, $|E^c_A|=|E^c_B|=1$, or $|E^c_A|=1$ and $|E^c_B|=2$.
	
	Suppose $|E^c_A|=0$.
	Since $G[A \cup C]$ cannot be 3-dependent,
	we must have $|E^c_B|=2$ and $|E^c[C]|=3$.
	By checking the possible non-edge combinations,
	we note that either no vertex of $C$ is a node and $G$ can be formed from the 3-rigid 3-circuit described in Figure \ref{fig:3k5s}(b) by a 1-extension (and hence we are done by Lemmas \ref{lem:1ext} and \ref{lem:d+5circcut}),
	or $C$ contains a node adjacent to only one vertex in $B$.
	As the only graph that satisfies the latter condition contains a double-banana subgraph (i.e.~the flexible 3-circuit formed by the deleted 2-sum of two copies of $K_5$), which can be found by deleting the node in $C$, we are done by Lemma \ref{lem: psd clique sum}.
	
	Now suppose $|E^c_A|=|E^c_B|=1$ (and hence $|E^c[C]|=3$).
	If the non-edges in $E_A^c$ and $E_B^c$ share an end then $G$ will contain a double-banana subgraph, so we may assume otherwise.
 	By checking all the remaining non-edge combinations,
 	we see that we can always 1-reduction to a node in $A$ that adds an edge between vertices in $C$,
	then apply another 1-reduction to a node in $B$ that adds an edge between vertices in $C$,
	and end up with the graph $H_3$.
	By Lemma \ref{lem: psd clique sum},
	$H_3$ has a PSD equilibrium stress,
	and by observation of the corresponding stress matrix we note it must have rank at least 2.
	The result now follows from Lemma \ref{lem:1ext}.
	
	Finally,
	suppose that $|E^c_A|=1$ and $|E^c_B|=2$ (and hence $|E^c[C]|=2$).
	By relabelling we may assume $c_1 c_2 , c_2 c_3 \in E^c[C]$ (i.e.~$c_1 c_3 \in E$) and $a_1$ is a node.
	If $a_1 c_2$ is a non-edge then $G[A \cup C]$ contains a copy of $K_5$,
	hence we may choose a non-edge $e \in \{c_1 c_2 , c_2 c_3\}$ so that both end points of $e$ are neighbours of $a_1$ in $G$.
	First suppose $G$ has two non-edges $b_i c_x, b_j c_x$ for distinct $b_i,b_j$.
	We must have $c_x \neq c_2$,
	as otherwise $G[B \cup C]$ will contain a copy of $K_5$.
	If $c_x$ is not a node, then $G$ is a deleted 3-sum of the globally 3-rigid graphs $K_6-\{e,f\}$ (see Theorem~\ref{thm:tibor}) and $K_5$;
	hence $G$ will have a PSD equilibrium stress by Lemma \ref{lem: psd clique sum}.
	If $c_x$ is a node then $G-c_x$ is the 2-sum of two copies of $K_5$;
	hence $G$ will have a PSD equilibrium stress by Lemma \ref{lem: psd clique sum}.
	Now suppose $G$ does not have two vertices in $B$ that are adjacent to the same vertex in the complement of $G$.
	Define $G' := G- a_1 +e$.
	If $E^c_B = \{b_i b_j , b_k c_\ell\}$ for distinct $i,j,k$,
	then the 1-reduction $G'-b_i + b_k c_\ell$ of $G'$ is $H_3$;
	hence the results follows from our previous observations of $H_3$ and Lemma \ref{lem:1ext}.
	If $E^c_B = \{b_i c_x, b_j c_y\}$ for distinct $i,j$ and distinct $x,y$,
	then $G'$ is the 3-sum of a copy of $K_5$ and the globally 3-rigid 3-circuit $K_{B \cup C} - E^c_B$ (see Theorem~\ref{thm:tibor}).
	Hence by Lemma \ref{lem: psd clique sum},
	$G$ has a PSD equilibrium stress.
\end{proof}

\begin{claim}
	If $A=\{a_1,a_2\}$ and $B=\{b_1,b_2,b_3,b_4\}$
	then $G$ has a 3-dimensional generic realization with a PSD equilibrium stress.
\end{claim}

\begin{proof}
	We note that $a_1a_2 \in E$ and $a_1c_i,a_2c_i \in E$ for each $i \in \{1,2,3\}$,
	as otherwise $a_1$ and $a_2$ would have a degree of 3 or less in $G$.
	If $G[C]$ has 3 edges, then $G[A\cup C]$ would be $K_5$ and so $G$ would not be a $3$-circuit.
	If $G[C]$ has 2 edges then $G$ is the 3-sum of $K_5$ and another 3-circuit with 7 vertices by \cite[Lemma 17(a)]{grasegger2020flexible},
	and hence $G$ will have a 3-dimensional generic realization with a PSD equilibrium stress by Lemmas \ref{lem: psd clique sum} and  \ref{lem:d+5circcut}.
	Suppose that $G[C]$ has either no edges or 1 edge.
	Applying a 1-reduction at either $a_1$ or $a_2$ and then applying a 0-reduction to the remaining vertex in $A$ is equivalent to deleting both $a_1$ and $a_2$ and adding an edge between two vertices in $C$. For brevity we refer to this process as an \emph{$A$-move}.
	
	Suppose that there exists an $A$-move that gives a graph with minimum degree 2.
	We can check all the possible cases where this happens by observing that $G$ has 6 non-edges with both ends in $B \cup C$,
	and at least two non-edges must have both ends in $C$.
	In every case we see that $G$ would contain either $K_5$ or $K_6- \{e,f\}$ as a subgraph, contradicting that $G$ is a 3-circuit.
	Hence we may assume that any $A$-move produces a graph with minimum degree~3.
	
	Now suppose that every $A$-move produces a graph with minimal degree 3.
	As $G$ is a 3-circuit,
	any vertex of degree 3 of $G'$ must lie in $C$.
	By checking the various assignments of non-edges between vertices in $B$ and $C$ we see that $G[C]$ must contain no edges;
	any possible graph where every $A$-move gives a graph with minimum degree 3 and $G[C]$ contains an edge would force $G$ to contain either $K_5$ or the 3-circuit $K_6-\{e,f\}$.
	This leaves the two possible 3-rigid 3-circuits given in Figure \ref{fig:amove}.
	The graph on the left has a vertex that we can apply a 1-reduction to so as to obtain the graph in Figure \ref{fig:3k5s}(b). We can verify that the claim holds for the remaining graph on the right in a similar manner as in Lemma~\ref{lemma: 6 regular}.\footnote{Let $G=(V,E)$ be the graph defined as follows. Put $V=\{a_1,a_2,b_1,b_2,b_3,b_4,c_1,c_2,c_3\}$ and   $E_2=\{a_1a_2,a_1c_1,a_1c_2,a_1c_3,a_2c_1,a_2c_2,a_2c_3,b_1b_2,b_1b_3,b_1b_4,b_2b_3,b_2b_4,b_3b_4,b_1c_1,b_1c_2,b_2c_1,b_2c_2,b_2c_3,b_3c_1,b_3c_3,b_4c_2,b_4c_3\}$. Define $(G,p)$ in $\mathbb{R}^3$ by putting $p(a_1)=(-42, -45, -40)$, $p(a_2)=(44, 48, 
  44)$, $p(b_1)=(9, -1, -7)$, $p(b_2)=(-8, -8, 
  3)$, $p(b_3)=(-1, -4, -5)$, $p(b_4)=(3, -7, 3)$, $p(c_1)=(1, -1, 9)$, $p(c_2)=(-3, -3, -4)$ and $p(c_3)=(-5, -10, -6)$. Given this realization it is simple for the reader to verify that $(G_i,p)$ is infinitesimally rigid in $\mathbb{R}^3$ and that the unique equilibrium stress of $(G,p)$ has a PSD stress matrix of rank 4. Since $G$ is not globally 3-rigid, it follows that a sufficiently nearby generic framework $(G,q)$ has a rank 4 PSD equilibrium stress.}

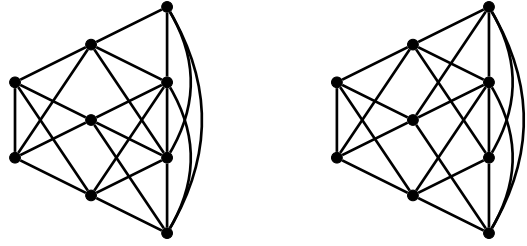
\begin{figure}[ht]
	\begin{center}
		\begin{tikzpicture}
			\node[vertex] (L1) at (-1,-0.5) {};
			\node[vertex] (L2) at (-1,0.5) {};
			\node[vertex] (C1) at (0,-1) {};
			\node[vertex] (C2) at (0,0) {};
			\node[vertex] (C3) at (0,1) {};
			\node[vertex] (R1) at (1,-1.5) {};
			\node[vertex] (R2) at (1,-0.5) {};
			\node[vertex] (R3) at (1,0.5) {};
			\node[vertex] (R4) at (1,1.5) {};
	
			\draw[edge] (L1)edge(C1);
			\draw[edge] (L1)edge(C2);
			\draw[edge] (L1)edge(C3);
			
			\draw[edge] (L2)edge(C1);
			\draw[edge] (L2)edge(C2);
			\draw[edge] (L2)edge(C3);
			
			\draw[edge] (L1)edge(L2);
			
			\draw[edge] (R1)edge(C1);
			\draw[edge] (R1)edge(C2);
			
			\draw[edge] (R2)edge(C1);
			\draw[edge] (R2)edge(C2);
			\draw[edge] (R2)edge(C3);
			
			\draw[edge] (R3)edge(C1);
			\draw[edge] (R3)edge(C2);
			\draw[edge] (R3)edge(C3);
			
			\draw[edge] (R4)edge(C3);
			
			\draw[edge] (R1)edge(R2);
			\draw[edge] (R2)edge(R3);
			\path[edge] (R3)edge(R4);	
			\path[edge] (R3) edge [bend left] node {} (R1);	
			\path[edge] (R4) edge [bend left] node {} (R2);	
			\path[edge] (R4) edge [bend left] node {} (R1);	
		\end{tikzpicture}\qquad\qquad%
		\begin{tikzpicture}
			\node[vertex] (L1) at (-1,-0.5) {};
			\node[vertex] (L2) at (-1,0.5) {};
			\node[vertex] (C1) at (0,-1) {};
			\node[vertex] (C2) at (0,0) {};
			\node[vertex] (C3) at (0,1) {};
			\node[vertex] (R1) at (1,-1.5) {};
			\node[vertex] (R2) at (1,-0.5) {};
			\node[vertex] (R3) at (1,0.5) {};
			\node[vertex] (R4) at (1,1.5) {};
	
			\draw[edge] (L1)edge(C1);
			\draw[edge] (L1)edge(C2);
			\draw[edge] (L1)edge(C3);
			
			\draw[edge] (L2)edge(C1);
			\draw[edge] (L2)edge(C2);
			\draw[edge] (L2)edge(C3);
			
			\draw[edge] (L1)edge(L2);
			
			\draw[edge] (R1)edge(C1);
			\draw[edge] (R1)edge(C2);
			
			\draw[edge] (R2)edge(C1);
			\draw[edge] (R2)edge(C3);
			
			\draw[edge] (R3)edge(C1);
			\draw[edge] (R3)edge(C2);
			\draw[edge] (R3)edge(C3);
			
			\draw[edge] (R4)edge(C2);
			\draw[edge] (R4)edge(C3);
			
			\draw[edge] (R1)edge(R2);
			\draw[edge] (R2)edge(R3);
			\path[edge] (R3)edge(R4);	
			\path[edge] (R3) edge [bend left] node {} (R1);	
			\path[edge] (R4) edge [bend left] node {} (R2);	
			\path[edge] (R4) edge [bend left] node {} (R1);	
		\end{tikzpicture}
	\end{center}
	\caption{The only two possible graphs that $G$ can be if every $A$-move gives a graph with minimal degree 3.}
	\label{fig:amove}
\end{figure}

\begin{figure}[ht]
	\begin{center}
		\begin{tikzpicture}
			\node[vertex] (L1) at (-1,-0.5) {};
			\node[vertex] (L2) at (-1,0.5) {};
			\node[vertex] (C1) at (0,-1) {};
			\node[vertex] (C2) at (0,0) {};
			\node[vertex] (C3) at (0,1) {};
			\node[vertex] (R1) at (1,-1.5) {};
			\node[vertex] (R2) at (1,-0.5) {};
			\node[vertex] (R3) at (1,0.5) {};
			\node[redvertex] (R4) at (1,1.5) {};
	
			\draw[edge] (L1)edge(C1);
			\draw[edge] (L1)edge(C2);
			\draw[edge] (L1)edge(C3);
			
			\draw[edge] (L2)edge(C1);
			\draw[edge] (L2)edge(C2);
			\draw[edge] (L2)edge(C3);

			\draw[edge] (L1)edge(L2);

			\draw[edge] (R1)edge(C1);
			\draw[edge] (R1)edge(C2);
			\draw[edge] (R1)edge(C3);
			
			\draw[edge] (R2)edge(C1);
			\draw[edge] (R2)edge(C2);
			\draw[edge] (R2)edge(C3);
			
			\draw[edge] (R3)edge(C1);
			\draw[edge] (R3)edge(C2);
			\draw[edge] (R3)edge(C3);
			
			\draw[edge] (R4)edge(C3);
			
			\draw[edge] (R1)edge(R2);
			\draw[edge] (R2)edge(R3);
			\path[edge] (R3)edge(R4);	
			\path[edge] (R4) edge [bend left] node {} (R1);	
			\path[edge] (R4) edge [bend left] node {} (R2);	
			\path[edge, red] (R3) edge [bend left] node {} (R1);	
			
			\node [rectangle,draw=white, fill=white] (a) at (0,-2) {(a)};
		\end{tikzpicture}\qquad\qquad%
		\begin{tikzpicture}
			\node[vertex] (L1) at (-1,-0.5) {};
			\node[vertex] (L2) at (-1,0.5) {};
			\node[vertex] (C1) at (0,-1) {};
			\node[vertex] (C2) at (0,0) {};
			\node[vertex] (C3) at (0,1) {};
			\node[bluevertex] (R1) at (1,-1.5) {};
			\node[vertex] (R2) at (1,-0.5) {};
			\node[vertex] (R3) at (1,0.5) {};
			\node[redvertex] (R4) at (1,1.5) {};
	
			\draw[edge] (L1)edge(C1);
			\draw[edge] (L1)edge(C2);
			\draw[edge] (L1)edge(C3);
			
			\draw[edge] (L2)edge(C1);
			\draw[edge] (L2)edge(C2);
			\draw[edge] (L2)edge(C3);

			\draw[edge] (L1)edge(L2);
			
			\path[edge] (C1)edge(C2);

			\draw[edge] (R1)edge(C1);
			\draw[edge] (R1)edge(C2);
			
			\draw[edge,cyan] (R2)edge(C1);
			\draw[edge] (R2)edge(C2);
			\draw[edge] (R2)edge(C3);
			
			\draw[edge] (R3)edge(C1);
			\draw[edge] (R3)edge(C2);
			\draw[edge] (R3)edge(C3);
			
			\draw[edge] (R4)edge(C1);
			\draw[edge] (R4)edge(C3);
			
			\draw[edge] (R1)edge(R2);
			\draw[edge] (R2)edge(R3);
			\path[edge] (R3)edge(R4);	
			\path[edge] (R4) edge [bend left] node {} (R2);	
			\path[edge] (R3) edge [bend left] node {} (R1);

			\path[edge,red] (C1) edge [bend left] node {} (C3);	
			
			\node [rectangle,draw=white, fill=white] (c) at (0,-2) {(b)};
		\end{tikzpicture}\qquad\qquad%
		\begin{tikzpicture}
			\node[vertex] (L1) at (-1,-0.5) {};
			\node[vertex] (L2) at (-1,0.5) {};
			\node[vertex] (C1) at (0,-1) {};
			\node[vertex] (C2) at (0,0) {};
			\node[vertex] (C3) at (0,1) {};
			\node[bluevertex] (R1) at (1,-1.5) {};
			\node[vertex] (R2) at (1,-0.5) {};
			\node[vertex] (R3) at (1,0.5) {};
			\node[redvertex] (R4) at (1,1.5) {};
	
			\draw[edge] (L1)edge(C1);
			\draw[edge] (L1)edge(C2);
			\draw[edge] (L1)edge(C3);
			
			\draw[edge] (L2)edge(C1);
			\draw[edge] (L2)edge(C2);
			\draw[edge] (L2)edge(C3);

			\draw[edge] (L1)edge(L2);
			
			\path[edge] (C1)edge(C2);
			\path[edge,red] (C2)edge(C3);
			
			\draw[edge] (R1)edge(C1);
			\draw[edge] (R1)edge(C2);
			
			\draw[edge] (R2)edge(C1);
			\draw[edge] (R2)edge(C2);
			\draw[edge] (R2)edge(C3);
			
			\draw[edge] (R3)edge(C1);
			\draw[edge] (R3)edge(C2);
			\draw[edge] (R3)edge(C3);
			
			\draw[edge] (R4)edge(C2);
			\draw[edge] (R4)edge(C3);
			
			\draw[edge] (R1)edge(R2);
			\draw[edge,cyan] (R2)edge(R3);
			\path[edge] (R3)edge(R4);	
			\path[edge] (R4) edge [bend left] node {} (R2);	
			\path[edge] (R3) edge [bend left] node {} (R1);	
			
			\node [rectangle,draw=white, fill=white] (c) at (0,-2) {(c)};
		\end{tikzpicture}
	\end{center}
	\caption{(a) Ignoring the red edge, a 3-rigid 3-circuit that can be formed from the 3-rigid 3-circuit in Figure \ref{fig:3k5s}(b) by a 1-extension. (b)--(c) Ignoring the red and blue edges, two graphs that can be formed from $H_3$ by two consecutive 1-extensions.}
	\label{fig:amove2}
\end{figure}
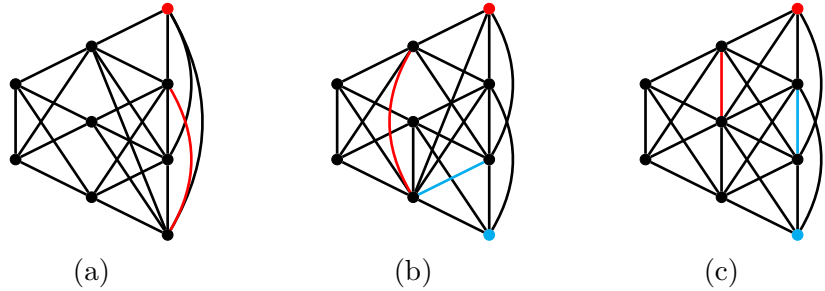

Hence, using the handshaking lemma, we may assume that $G$ has an $A$-move that produces a graph $G'$ with minimal degree 4.
	If $G'$ is globally 3-rigid the result follows from Theorem \ref{thm: gen ur} and Lemma \ref{lem:1ext},
	so we may suppose $G'$ is not globally 3-rigid.
	By an easy case-by-case check of graphs on 7 vertices with minimum degree 4 and $16 (= 3\cdot 7 -5)$ edges, we see that we must have $G'=H_3$.
	As we are assuming $G$ has no separating set of size 3 with more than 1 edge,
	it follows that $G$ must be one of the three graphs depicted in Figure \ref{fig:amove2};
	we can see this by systematically applying reverse $A$-moves to $H_3$.
	We can obtain the 3-rigid 3-circuit in Figure~\ref{fig:3k5s}(b) from the graph in Figure~\ref{fig:amove2}(a) by applying a 1-reduction at the red vertex and adding the red edge, hence it will also have a 3-dimensional PSD equilibrium stress by Lemmas~\ref{lem:1ext} and~\ref{lem:d+5circcut}.
	For the other two graphs in Figure \ref{fig:amove2}, we can apply a 1-reduction at the red vertex to add the red edge and then a 1-reduction at the blue vertex that adds the blue edge to obtain $H_3$. Hence Lemmas \ref{lem:1ext} and \ref{lem:d+5circcut} complete the proof.
\end{proof}
The above claims cover all possibilities and hence complete the proof.
\end{proof}

\section*{Acknowledgments}

This paper arose as part of the Fields Institute 
Thematic Program on Geometric constraint systems, framework rigidity, and distance geometry.

DIB was partially supported by a Mathematical Sciences
Postdoctoral Research Fellowship from the US NSF, grant DMS-1802902.
SD was supported by the Heilbronn Institute for Mathematical Research and the Austrian Science Fund (FWF): P31888.
AN was partially supported by EPSRC grant EP/W019698/1. SJG was partially supported by US NSF grant 
DMS-1564473. MS was partially supported by US
NSF grant DMS-1564480 and US NSF grant DMS-1563234.

\appendix

\section{1-extensions that preserve PSD equilibrium stresses}
\label{sec:app}

We prove Lemma \ref{lem:1ext}.
We first require the following technical result from 
 \cite[Lemma 4.9]{cgt2} and the discussion around it.
 
\begin{lemma}\label{lem: psd stress sign}
    Let $(G,p)$ be a $d$-dimensional framework and $\omega$ a PSD 
    equilibrium stress of $(G,p)$ and let $xy$ be an edge of $G$ so that 
    $\omega_{xy} > 0$.  
    Then there is a non-singular 
    projective transformation $T$ on $\RR^d$ so that $(G,T(p))$
    has a PSD equilibrium stress $\psi$ so that $\psi_{xy} < 0$.
\end{lemma}

\begin{proof}[Proof of Lemma \ref{lem:1ext}]
    Let $xy$ be the edge deleted during the 1-extension and $z$ be the vertex added to form $G'$.
    Denote the neighbours of $z$ that are neither $x$ nor $y$ by $v_1,\ldots, v_{d-1}$.
    For each $t \in \mathbb{R}$,
    define $p^t$ to be the realization of $G'$ in $\mathbb{R}^d$ where $p^t_v = p_v$ for all $v \in V$ and $p^t_z = tp_x + (1-t)p_y$.
    Also for each $t \in \mathbb{R}$ define the map $\omega^t : E' \rightarrow \mathbb{R}$ by setting $\omega^t_e = \omega_e$ for all $e \in E$,
    $\omega^t_{xz} = (1-t)^{-1}\omega_{xy}$, $\omega^t_{yz} = t^{-1}\omega_{xy}$ and $\omega^t_{z v_i} = 0$ for all $i \in \{1,\ldots, d-1\}$.
    Note that for each $t \in \mathbb{R}$, $\omega^t$ is an equilibrium stress of $(G',p^t)$.
    Since $G$ is a $d$-circuit,
    $\omega_{xy} \neq 0$.
    Hence by Lemma \ref{lem: psd stress sign},
    we may suppose that $\omega_{xy}= -1$.

   Given $\Omega$ is the equilibrium stress matrix associated with $\omega$, 
    the equilbrium stress matrix associated with $\omega^t$ is the matrix
    \begin{align*}
        \Omega^t = 
        \begin{bmatrix}
            M_t & 0_{3 \times (|V|-2)} \\
            0_{(|V|-2) \times 3} & 0_{(|V|-2) \times (|V|-2)}
        \end{bmatrix}
        + 
        \begin{bmatrix}
            0 & 0_{1 \times |V|} \\
            0_{|V| \times 1} & \Omega
        \end{bmatrix},
    \end{align*}
    where
    \begin{align*}
        M_t =
        \begin{bmatrix}
            \frac{-1}{t(1-t)} & \frac{1}{1-t} & \frac{1}{t}\\
            \frac{1}{1-t} & \frac{-t}{1-t} & -1\\
            \frac{1}{t} & -1 & \frac{t-1}{t}
        \end{bmatrix}.
    \end{align*}
    Suppose $a \in \mathbb{R}^E$ is an element of the kernel of $\Omega^t$.
    By observing the $z$-coordinate of $\Omega^t a$ we see that $a_z = ta_x + (1-t) a_y$.
    By observing the $x$- and $y$-coordinates of $\Omega^t a$ with this substitution we see that $\sum_{v \in N_G(x)} \omega^t(a_x - a_v) = 0$ and $\sum_{v \in N_G(y)} \omega^t(a_y - a_v) = 0$.
    Thus the nullity of $\Omega^t$ is equal to the nullity of $\Omega$,
    and so $\rank \Omega^t = \rank \Omega +1 \geq |V'|-d-2$ for each $t \neq 0,1$.
    As $\rank M_t = 1$ for all $t \neq 0,1$ and
    \begin{align*}
        M_t \rightarrow 
        \begin{bmatrix}
            0 & 0 & 0\\
            0 & 1 & -1\\
            0 & -1 & 1
        \end{bmatrix}
        \qquad \text{as } t \rightarrow \infty,
    \end{align*}
    the matrix $M_t$ is PSD for sufficiently large $t$.
    Hence we can now fix some $T >1$ such that $\Omega^T$ is PSD.

    By Lemma \ref{lem:0ext},
    $(G',p^t)$ is infinitesimally rigid,
    hence there exists an open neighbourhood $U$ of $p^T$ where for each $p' \in U$, the framework $(G',p')$ has exactly one (up to scalar multiplication) equilibrium stress and the rigidity matrix $R(G',p')$ has maximal rank over all realizations.
    It follows that we can define a continuous map $\lambda:U \rightarrow \mathbb{R}^{E'}$ such that $\lambda(q)$ is an equilibrium stress of $(G',q)$ and $\lambda(p^T) = \omega^T$.
    Suppose that an equilibrium stress $\lambda(q)$ has rank $|V'|-d-1$ for some $q \in U$.
    By \cite[Theorem~1.3]{Bob05}, $G'$ is globally $d$-rigid.
    Hence there exists a generic framework $(G',p')$ with a PSD equilibrium stress of rank $|V'|-d-1$ by Theorem \ref{thm: gen ur}. 
    Suppose instead that the equilibrium stress $\lambda(q)$ has rank at most $|V'|-d-2$ for all $q \in U$.
    Then the rank of the equilibrium stress $\lambda(p^T) = \omega^T$ is maximal over $U$.
    As the rank function is lower semi-continuous,
    there exists an open neighbourhood $U' \subset U$ of $p^T$ where each equilibrium stress $\lambda(q)$ with $q \in U'$ has rank $|V'|-d-2$.
    
    Define, for each $i \in \{1,\ldots, |V'|\}$, the continuous map $\mu_i :U' \rightarrow \mathbb{R}^{|V|-d-2}$ which maps a realization $q \in U'$ to its $i$-th highest eigenvalue of the equilibrium stress matrix associated to $\lambda (q)$.
    Since $\omega^T$ is PSD with rank $|V'|-d-2$,
    $\mu_i(p^T) >0$ for all $i \leq |V'|-d-2$ and $\mu_i(p^T) =0$ for all $i > |V'|-d-2$.
    By the continuity of the $\mu_i$ maps,
    there exists a sufficiently close generic realization $p' \in U'$ where $\mu_i(p') >0$ for all $i \leq |V'|-d-2$.
    Since the rank of $\lambda(p')$ is $|V'|-d-2$,
    $\mu_i(q) =0$ for all $i > |V'|-d-2$ also.
    Hence the generic framework $(G',p')$ has a PSD equilibrium stress of rank $|V'|-d-2$.
\end{proof}

\bibliographystyle{plainurl} 
\bibliography{mltRigid}      
\end{document}